\documentclass[11pt]{amsart}

\usepackage{lmodern} 
\usepackage[T1]{fontenc}

\theoremstyle{plain}
\newtheorem{theorem}                 {Theorem}      [section]
\newtheorem{proposition}  [theorem]  {Proposition}
\newtheorem{corollary}    [theorem]  {Corollary}
\newtheorem{conjecture}   []  {Conjecture}
\newtheorem{problem}   []  {Problem}

\newtheorem{lemma}        [theorem]  {Lemma}

\theoremstyle{definition}

\newtheorem{remark}       [theorem]  {Remark}

\numberwithin{equation}{section}

\DeclareMathOperator{\trace}{trace}
\DeclareMathOperator{\grad}{grad}

\DeclareMathOperator{\Div}{div}
\DeclareMathOperator{\Id}{Id}

\DeclareMathOperator{\ricci}{Ricci}

\makeindex

\begin{document}

\title[Classification of Biharmonic submanifolds]
{New results toward the classification of Biharmonic submanifolds in $\mathbb{S}^{n}$}

\date{\today}

\author{A.~Balmu\c s}
\address{Faculty of Mathematics, ``Al.I.~Cuza'' University of Iasi\\
\newline
Bd. Carol I Nr. 11 \\
700506 Iasi, ROMANIA}
\email{adina.balmus@uaic.ro}
\author{S. Montaldo}
\address{Universit\`a degli Studi di Cagliari\\
Dipartimento di Matematica\\
\newline
Via Ospedale 72\\
09124 Cagliari, ITALIA} \email{montaldo@unica.it}
\author{C. Oniciuc}
\address{Faculty of Mathematics, ``Al.I.~Cuza'' University of Iasi\\
\newline
Bd. Carol I Nr. 11 \\
700506 Iasi, ROMANIA} \email{oniciucc@uaic.ro}


\thanks{The first author was supported by Grant POSDRU/89/1.5/S/49944, Romania. The second author was supported by Contributo d'Ateneo, University of Cagliari, Italy. The third author was supported  by a grant of the Romanian National Authority for
Scientific Research, CNCS -- UEFISCDI, project number PN-II-RU-TE-2011-3-0108; and by Regione Autonoma della Sardegna, Visiting Professor Program.}

\begin{abstract}
We prove some new rigidity results for proper biharmonic immersions in ${\mathbb S}^n$ of the following types: Dupin hypersurfaces; hypersurfaces, both compact and non-compact, with bounded norm of the second fundamental form; hypersurfaces satisfying intrinsic properties; PMC submanifolds;  parallel submanifolds.

\end{abstract}


\maketitle

\section{Introduction}

Let $\varphi:M\to (N,h)$ be an immersion  of a manifold $M$ into a
Riemannian manifold $(N,h)$. We say that $\varphi$ is {\em
biharmonic}, or $M$ is a {\em biharmonic submanifold}, if its mean curvature vector
field $H$ satisfies the following equation
\begin{eqnarray}\label{eq: bih_eq}
\tau_2(\varphi)=- m\left(\Delta H + \trace{R^{N}}(
d\varphi(\cdot), H) d\varphi(\cdot)\right)=0,
\end{eqnarray}
where $\Delta$ denotes the rough Laplacian on sections of the
pull-back bundle $\varphi^{-1}(TN)$ and $R^N$ denotes the curvature
operator on $(N,h)$. The section $\tau_2(\varphi)$ is called the
{\em bitension field}.

When $M$ is compact, the biharmonic condition arises from a
variational problem for maps: for an arbitrary smooth map
$\varphi:(M,g)\to (N,h)$ we define
$$
E_{2}\left( \varphi \right) = \frac{1}{2} \int_{M} |\tau(\varphi)|^{2}\, v_{g},
$$
where $\tau(\varphi)=\trace\nabla d\varphi$ is the {\it tension field}. The functional $E_2$ is called  the {\em bienergy functional}. When $\varphi:(M,\varphi^{\ast}h)\to (N,h)$ is an immersion, the tension
field has the expression $\tau(\varphi)=mH$ and \eqref{eq: bih_eq} is equivalent to $\varphi$ being a critical point of $E_2$.

Obviously, any minimal immersion ($H=0$) is biharmonic. The non-harmonic biharmonic immersions  are called {\it proper biharmonic}.

The study of proper biharmonic submanifolds is nowadays becoming a
very active subject and its popularity initiated with the
challenging conjecture of B-Y.~Chen (see the recent book \cite{C11}): {\em any biharmonic submanifold
in the Euclidean space is minimal}.

Chen's conjecture was generalized to: {\em any biharmonic submanifold in a
Riemannian manifold with non-positive sectional curvature is
minimal}, but this was proved not to hold. Indeed, in \cite{OT10}, Y.-L.~Ou and L.~Tang
constructed examples of proper biharmonic hypersurfaces
in a $5$-dimensional space of non-constant negative sectional
curvature.

Yet, the conjecture is still open in its full generality for ambient
spaces with constant non-positive sectional curvature, although it
was proved to be true in numerous cases when additional geometric
properties for the submanifolds were assumed (see, for example,
\cite{BMO08,CMO02,C91,D98,D92,HV95}).

By way of contrast, as we shall detail in Section~\ref{sec:
bih-sub}, there are several families of examples of proper
biharmonic submanifolds in the $n$-dimensional unit Euclidean sphere
$\mathbb{S}^{n}$.  For simplicity we shall denote these classes by {\bf B1}, {\bf B2}, {\bf B3} and {\bf B4}.

 The goal of this paper is to continue the study of
proper biharmonic submanifolds in $\mathbb{S}^{n}$ in order to achieve their classification.
This program was initiated for the very first time in \cite{J86} and then developed
in \cite{BMO12} -- \cite{BO09}, \cite{CMO02,CMO01,NU11,NU12, O02}.

In the following, by a rigidity result for proper biharmonic submanifolds we mean:\\
{\em find under what conditions a proper biharmonic submanifold in  ${\mathbb S}^n$ is one of the main examples {\bf B1}, {\bf B2}, {\bf B3} and {\bf B4}}.

We prove rigidity results for the following types of submanifolds in  ${\mathbb S}^n$: Dupin hypersurfaces; hypersurfaces, both compact and non-compact, with bounded norm of the second fundamental form; hypersurfaces satisfying intrinsic geometric properties; PMC submanifolds;  parallel submanifolds.

Moreover, we include in this paper two results of  J.H.~Chen published in \cite{C93}, in Chinese. We give a complete proof of these results using the invariant formalism and shortening the original proofs.

\vspace{2mm}

{\bf Conventions.}
Throughout this paper all manifolds, metrics, maps are assumed to be smooth, i.e. $C^\infty$. All manifolds are assumed to be connected. The following  sign conventions are used
$$
\Delta V=-\trace\nabla^2 V\,,\qquad R^N(X,Y)=[\nabla_X,\nabla_Y]-\nabla_{[X,Y]},
$$
where $V\in C(\varphi^{-1}(TN))$ and $X,Y\in C(TN)$.
Moreover, the Ricci and scalar curvature $s$ are defined as
$$
\langle \ricci(X),Y\rangle=\ricci(X,Y)=\trace (Z\to R(Z,X)Y)), \quad s=\trace \ricci,
$$
where $X,Y,Z\in C(TN)$.

\vspace{2mm}

{\bf Acknowledgements.}
The authors would like  to thank professor Jiaping Wang for some helpful discussions and  Juan Yang for the accurate translation of \cite{C93}. The third author would like to thank the Department of Mathematics and Informatics of the University of Cagliari for the warm hospitality.

\section{Biharmonic immersions in   ${\mathbb S}^n$}\label{sec: bih-sub}

The key ingredient in the study of biharmonic submanifolds is the
splitting of the bitension field with respect to its normal and
tangent components. In the case when the ambient space is the unit Euclidean sphere we have the following characterization.

\begin{theorem}[\cite{C84, O02}]\label{th: bih subm S^n}
An immersion $\varphi:M^m\to\mathbb{S}^n$ is biharmonic if and only if
\begin{equation}\label{eq: caract_bih_spheres}
\left\{
\begin{array}{l}
\ \Delta^\perp {H}+\trace B(\cdot,A_{H}\cdot)-m\,{H}=0,
\vspace{2mm}
\\
\ 2\trace A_{\nabla^\perp_{(\cdot)}{H}}(\cdot)
+\dfrac{m}{2}\grad {|H|}^2=0,
\end{array}
\right.
\end{equation}
where $A$ denotes the Weingarten operator, $B$ the second
fundamental form, ${H}$ the mean curvature vector field, $|H|$ the mean curvature function,
$\nabla^\perp$ and $\Delta^\perp$ the connection and the Laplacian
in the normal bundle of $\varphi$, respectively.
\end{theorem}

In the codimension one case, denoting by $A=A_\eta$ the shape operator with respect to a (local) unit section $\eta$ in the normal bundle and
putting   $f=(\trace A)/m$, the above result reduces to the following.
\begin{corollary}[\cite{O02}]\label{cor: caract_hypersurf_bih}
Let $\varphi:M^m\to\mathbb{S}^{m+1}$ be an orientable hypersurface. Then $\varphi$ is biharmonic if and only if
\begin{equation}\label{eq: caract_bih_hipersurf_spheres}
\left\{
\begin{array}{l}
{\rm (i)}\quad \Delta f=(m-|A|^2) f,
\\ \mbox{} \\
{\rm (ii)}\quad A(\grad f)=-\dfrac{m}{2}f\grad f.
\end{array}
\right.
\end{equation}
\end{corollary}

A special class of immersions in $\mathbb{S}^n$ consists of the parallel mean curvature immersions (PMC), that is immersions such that $\nabla^{\perp}H=0$. For this class of immersions Theorem~\ref{th: bih subm S^n} reads as follows.

\begin{corollary}[\cite{BO12}]\label{th: caract_bih_pmc}
Let $\varphi:M^m\to\mathbb{S}^n$ be a PMC  immersion. Then $\varphi$ is biharmonic if and only if
\begin{equation}\label{eq: caract_bih_Hparallel_I}
\trace B(A_H(\cdot),\cdot)=mH,
\end{equation}
or equivalently,
\begin{equation}\label{eq: caract_bih_Hparallel_II}
\left\{
\begin{array}{ll}
\langle A_H, A_\xi\rangle=0,\quad \forall\xi\in C(NM)\, \text{with}\,\, \xi\perp H,
\\ \mbox{} \\
|A_H|^2=m|H|^2,
\end{array}
\right.
\end{equation}
where $NM$ denotes the normal bundle of $M$ in $\mathbb{S}^n$.
\end{corollary}

We now list the main examples of proper biharmonic immersions in $\mathbb{S}^n$.

\begin{list}{\labelitemi}{\leftmargin=2em\itemsep=1.5mm\topsep=0mm}
\item[{\bf B1}.] The canonical inclusion of the small hypersphere
\begin{equation*}\label{eq: small_hypersphere}
\mathbb{S}^{n-1}(1/\sqrt 2)=\left\{(x,1/\sqrt 2)\in\mathbb{R}^{n+1}: x\in \mathbb{R}^n, |x|^2=1/2\right\}\subset\mathbb{S}^{n}.
\end{equation*}
\item[{\bf B2}.] The canonical inclusion of the standard (extrinsic) products of spheres
\begin{equation*}\label{eq: product_spheres}
\mathbb{S}^{n_1}(1/\sqrt 2)\times\mathbb{S}^{n_2}(1/\sqrt 2)=\left\{(x,y)\in\mathbb{R}^{n_1+1}\times\mathbb{R}^{n_2+1}, |x|^2=|y|^2=1/2\right\}\subset\mathbb{S}^{n},
\end{equation*}
$n_1+n_2=n-1$ and $n_1\neq n_2$.
\item[{\bf B3}.] The maps $\varphi=\imath\circ\phi:M\to \mathbb{S}^n$, where $\phi:M\to \mathbb{S}^{n-1}(1/\sqrt 2)$ is a minimal immersion, and $\imath:\mathbb{S}^{n-1}(1/\sqrt 2)\to\mathbb{S}^n$ denotes the canonical inclusion.

\item[{\bf B4}.] The maps $\varphi=\imath\circ(\phi_1\times\phi_2): M_1\times M_2\to \mathbb{S}^n$, where $\phi_i:M_i^{m_i}\to\mathbb{S}^{n_i}(1/\sqrt 2)$, $0 < m_i \leq n_i$, $i=1,2$, are minimal immersions, $m_1\neq m_2$, $n_1+n_2=n-1$, and $\imath:\mathbb{S}^{n_1}(1/\sqrt 2)\times\mathbb{S}^{n_2}(1/\sqrt 2)\to \mathbb{S}^n$ denotes the canonical inclusion.
\end{list}

\begin{remark}
\begin{itemize}
\item[(i)] The proper biharmonic  immersions of class {\bf B3} are pseudo-umbilical, i.e. $A_H=|H|^2\Id$, have parallel mean curvature vector field and mean curvature $|H|=1$. Clearly, $\nabla A_H=0$.

\item[(ii)] The proper biharmonic  immersions of class {\bf B4} are no longer pseudo-umbilical, but still have parallel mean curvature vector field and their mean curvature is $|H|={|m_1-m_2|}/{m}\in(0,1)$, where $m=m_1+m_2$. Moreover, $\nabla A_H=0$ and the principal curvatures in the direction of $H$, i.e. the eigenvalues of $A_H$, are constant on $M$ and given by $\lambda_1=\ldots=\lambda_{m_1}=({m_1-m_2})/{m}$, $\lambda_{m_1+1}=\ldots=\lambda_{m_1+m_2}=-({m_1-m_2})/{m}$. Specific B4 examples were given by W.~Zhang in \cite{Z11} and generalized in \cite{BMO08a, WW12}.
\end{itemize}
\end{remark}

When a biharmonic immersion has constant mean curvature (CMC) the following bound for $|H|$ holds.

\begin{theorem}[\cite{O03}]\label{teo:h=cst-b3}
Let $\varphi:M\to\mathbb{S}^n$ be a CMC proper biharmonic  immersion. Then $|H|\in(0,1]$, and $|H|=1$ if and only if $\varphi$ induces a minimal immersion of $M$ into $\mathbb{S}^{n-1}(1/\sqrt 2)\subset\mathbb{S}^n$, that is $\varphi$ is {\bf B3}.
\end{theorem}

\section{Biharmonic hypersurfaces in spheres}

The first case to look at is that of CMC proper biharmonic hypersurfaces in $\mathbb{S}^{m+1}$.

\begin{theorem}[\cite{BMO08, BO12}]
Let $\varphi:M^m\to \mathbb{S}^{m+1}$ be a CMC proper biharmonic hypersurface. Then
\begin{itemize}
\item[(i)] $|A|^2=m$;
\item[(ii)] the scalar curvature $s$ is constant and positive, $s=m^2(1+|H|^2)-2m$;
\item[(iii)] for $m>2$, $|H|\in(0,({m-2})/{m}]\cup\{1\}$. Moreover, $|H|=1$ if and only if $\varphi(M)$ is an open subset of the small hypersphere $\mathbb{S}^m(1/\sqrt 2)$, and $|H|=({m-2})/{m}$ if and only if $\varphi(M)$ is an open subset of the standard product $\mathbb{S}^{m-1}(1/\sqrt 2)\times\mathbb{S}^1(1/\sqrt 2)$.
\end{itemize}
\end{theorem}

\begin{remark}
In the minimal case the condition $|A|^2=m$ is exhaustive. In fact a minimal hypersurface in $\mathbb{S}^{m+1}$ with $|A|^2=m$ is a minimal standard product of spheres (see \cite{CdCK70, L69}). We point out that the full classification of CMC hypersurfaces in $\mathbb{S}^{m+1}$  with $|A|^2=m$, therefore biharmonic, is not known.
\end{remark}

\begin{corollary}
Let $\varphi:M^m\to \mathbb{S}^{m+1}$ be a complete  proper biharmonic hypersurface.
\begin{itemize}
\item[(i)] If  $|H|=1$, then  $\varphi(M)=\mathbb{S}^{m}(1/\sqrt 2)$ and $\varphi$ is an embedding.
\item[(ii)] If  $|H|=({m-2})/{m}$, $m>2$, then   $\varphi(M)=\mathbb{S}^{m-1}(1/\sqrt 2)\times\mathbb{S}^1(1/\sqrt 2)$ and the universal cover of $M$ is $\mathbb{S}^{m-1}(1/\sqrt 2)\times\mathbb{R}$.
\end{itemize}
\end{corollary}

As a direct consequence of \cite[Theorem 2]{NS69} we have the following result.

\begin{theorem}
Let $\varphi:M^m\to \mathbb{S}^{m+1}$ be a CMC proper biharmonic hypersurface. Assume that $M$ has non-negative sectional curvature. Then $\varphi(M)$ is either an open part of $\mathbb{S}^{m}(1/\sqrt 2)$, or an open part of $\mathbb{S}^{m_1}(1/\sqrt 2)\times \mathbb{S}^{m_2}(1/\sqrt 2)$, $m_1+m_2=m$, $m_1\neq m_2$.
\end{theorem}

In the following we shall no longer assume that the biharmonic hypersurfaces have constant mean curvature, and we shall split our study in three cases. In Case 1 we shall study the proper biharmonic hypersurfaces with respect to the number of their distinct principal curvatures, in Case 2 we shall study them with respect to $|A|^2$ and $|H|^2$, and in Case 3 the study will be done with respect to the sectional and Ricci curvatures of the hypersurface.

\subsection{Case 1}

Obviously, if $\varphi:M^m\to \mathbb{S}^{m+1}$ is an umbilical proper biharmonic hypersurface in $\mathbb{S}^{m+1}$, then $\varphi(M)$ is an open part of $\mathbb{S}^{m}(1/\sqrt 2)$.

When the  hypersurface has at most two or exactly three distinct principal curvatures everywhere we obtain the following rigidity results.

\begin{theorem}[\cite{BMO08}]\label{th: hypersurf_2curv}
Let $\varphi:M^m\to \mathbb{S}^{m+1}$ be a hypersurface. Assume that $\varphi$ is proper biharmonic with at most two distinct principal curvatures everywhere. Then $\varphi$ is CMC and $\varphi(M)$ is either an open part of $\mathbb{S}^{m}(1/\sqrt 2)$, or an open part of $\mathbb{S}^{m_1}(1/\sqrt 2)\times \mathbb{S}^{m_2}(1/\sqrt 2)$, $m_1+m_2=m$, $m_1\neq m_2$. Moreover, if $M$ is complete, then either
$\varphi(M)=\mathbb{S}^{m}(1/\sqrt 2)$ and $\varphi$ is an embedding, or $\varphi(M)=\mathbb{S}^{m_1}(1/\sqrt 2)\times \mathbb{S}^{m_2}(1/\sqrt 2)$, $m_1+m_2=m$, $m_1\neq m_2$ and $\varphi$ is an embedding when $m_1\geq 2$ and $m_2\geq 2$.
\end{theorem}

\begin{theorem}[\cite{BMO08}]\label{teo:quasi-umbilicall-conformally-flat}
Let $\varphi:M^m\to \mathbb{S}^{m+1}$, $m\geq 3$, be a  proper biharmonic hypersurface. The following statements are equivalent:
\begin{itemize}
\item[(i)] $\varphi$ is quasi-umbilical,
\item[(ii)] $\varphi$ is conformally flat,
\item[(iii)] $\varphi(M)$ is an open part of $\mathbb{S}^m(1/\sqrt 2)$ or of $\mathbb{S}^{m-1}(1/\sqrt 2)\times \mathbb{S}^{1}(1/\sqrt 2)$.
\end{itemize}
\end{theorem}

It is well known that, if $m\geq 4$, a hypersurface $\varphi:M^m\to \mathbb{S}^{m+1}$ is quasi-umbilical if and only if it is conformally flat. From Theorem~\ref{teo:quasi-umbilicall-conformally-flat}
we see that under the biharmonicity hypothesis the equivalence remains true  when $m=3$.

\begin{theorem}[\cite{BMO10}]\label{th: hypersurf_3curv}
There exist no compact CMC proper biharmonic hypersurfaces $\varphi:M^m\to \mathbb{S}^{m+1}$ with three distinct principal curvatures everywhere.
\end{theorem}

In particular, in the low dimensional cases, Theorem~\ref{th: hypersurf_2curv}, Theorem~\ref{th: hypersurf_3curv} and a result of S.~Chang (see \cite{CH93}) imply the following.

\begin{theorem}[\cite{CMO01,BMO10}]
Let $\varphi:M^m\to \mathbb{S}^{m+1}$ be a proper biharmonic hypersurface.
\begin{itemize}
\item[(i)] If $m=2$, then $\varphi(M)$ is an open part of $\mathbb{S}^{2}(1/\sqrt 2)\subset\mathbb{S}^3$.
\item[(ii)] If $m=3$ and $M$ is compact, then $\varphi$ is CMC and $\varphi(M)=\mathbb{S}^{3}(1/\sqrt 2)$ or $\varphi(M)=\mathbb{S}^{2}(1/\sqrt 2)\times\mathbb{S}^{1}(1/\sqrt 2)$.
\end{itemize}
\end{theorem}

We recall that an orientable hypersurface $\varphi:M^m\to\mathbb{S}^{m+1}$ is
said to be {\it isoparametric} if it has constant principal curvatures  or, equivalently, the number $\ell$ of distinct principal curvatures $k_1 > k_2>\cdots
> k_\ell$ is constant on $M$  and the $k_i$'s are constant. The distinct principal curvatures have constant multiplicities $m_1, \ldots,m_\ell$, $m = m_1 + m_2 + \ldots + m_\ell$.

In \cite{IIU08}, T.~Ichiyama, J.I.~Inoguchi and H.~Urakawa classified the proper biharmonic isoparametric hypersurfaces in spheres.

\begin{theorem}[\cite{IIU08}]\label{teo:isoparametric}
Let $\varphi:M^m\to \mathbb{S}^{m+1}$ be an orientable isoparametric hypersurface. If $\varphi$ is proper biharmonic, then $\varphi(M)$ is either an open part of $\mathbb{S}^m(1/\sqrt2)$, or an open part of
$\mathbb{S}^{m_1}(1/\sqrt2)\times\mathbb{S}^{m_2}(1/\sqrt2)$, $m_1+m_2=m$,
$m_1\neq m_2$.
\end{theorem}

An orientable hypersurface $\varphi:M^m\to\mathbb{S}^{m+1}$ is
said to be a {\it proper Dupin hypersurface} if the number $\ell$ of  distinct principal curvatures is constant on $M$ and each principal curvature function is constant along its corresponding principal directions.

\begin{theorem}\label{th: Dupin_bih_CMC}
Let $\varphi:M^m\to \mathbb{S}^{m+1}$ be an orientable proper Dupin hypersurface. If $\varphi$ is proper biharmonic, then $\varphi$ is CMC.
\end{theorem}

\begin{proof}
As $M$ is orientable, we fix $\eta\in C(NM)$ and denote $A=A_\eta$ and $f=(\trace A)/m$. Suppose that $f$ is not constant. Then there exists an open subset $U\subset M$ such that $\grad f\neq 0$ at every point of $U$. Since $\varphi$ is proper biharmonic, from \eqref{eq: caract_bih_hipersurf_spheres} we get that $-{mf}/{2}$ is a principal curvature with principal direction $\grad f$. Since the hypersurface is proper Dupin, by definition, $\grad f(f)=0$, i.e. $\grad f=0$ on $U$, and we come to a contradiction.
\end{proof}

\begin{corollary}
Let $\varphi:M^m\to \mathbb{S}^{m+1}$ be an orientable proper Dupin hypersurface with $\ell\leq 3$. If $\varphi$ is proper biharmonic, then $\varphi(M)$ is either an open part of $\mathbb{S}^m(1/\sqrt2)$, or an open part of
$\mathbb{S}^{m_1}(1/\sqrt2)\times\mathbb{S}^{m_2}(1/\sqrt2)$, $m_1+m_2=m$,
$m_1\neq m_2$.
\end{corollary}

\begin{proof}
Taking into account Theorem~\ref{th: hypersurf_2curv}, we only have to prove that there exist no proper biharmonic proper Dupin hypersurfaces with $\ell=3$. Indeed, by Theorem~\ref{th: Dupin_bih_CMC}, we conclude that $\varphi$ is CMC. By a result in \cite{BMO12}, $\varphi$  is of type $1$ or of type $2$, in the sense of B.-Y.~Chen. If $\varphi$ is of type $1$, we must have $\ell=1$ and we get a contradiction. If $\varphi$ is of type $2$, since $\varphi$ is proper Dupin with $\ell=3$, from Theorem~9.11 in \cite{C96}, we get that $\varphi$ is isoparametric. But, from
Theorem~\ref{teo:isoparametric}, proper biharmonic isoparametric hypersurfaces must have $\ell\leq 2$.

\end{proof}

\subsection{Case 2}
The simplest result is the following.
\begin{proposition}\label{pro: a-compact}
Let $\varphi:M^m\to \mathbb{S}^{m+1}$ be a compact hypersurface. Assume that $\varphi$ is proper biharmonic with nowhere zero mean curvature vector field and $|A|^2\leq m$, or $|A|^2\geq m$. Then $\varphi$ is CMC and $|A|^2=m$.
\end{proposition}
\begin{proof}
As $H$ is nowhere zero, we can consider $\eta=H/|H|$ a global unit section in the normal bundle $NM$ of $M$ in $\mathbb{S}^{m+1}$.
Then, on $M$,
$$
\Delta f=(m-|A|^2)f,
$$
where $f=(\trace A)/m=|H|$. Now, as $m-|A|^2$ does not change sign, from the maximum principle we get $f=$ constant and $|A|^2=m$.
\end{proof}

In fact, Proposition~\ref{pro: a-compact} holds without the hypothesis ``$H$ nowhere zero''. In order to prove this we shall consider the cases $|A|^2\geq m$ and $|A|^2\leq m$, separately.

\begin{proposition}\label{prop: |B|>m}
Let $\varphi:M^m\to \mathbb{S}^{m+1}$ be a compact hypersurface. Assume that $\varphi$ is proper biharmonic and $|A|^2\geq m$. Then $\varphi$ is CMC and $|A|^2=m$.
\end{proposition}
\begin{proof}
Locally,
$$
\Delta f=(m-|A|^2)f,
$$
where $f=(\trace A)/m$, $f^2=|H|^2$,
and therefore
$$
\frac{1}{2}\Delta f^2=(m-|A|^2)f^2-|\grad f|^2\leq 0.
$$
As $f^2$, $|A|^2$ and  $|\grad f|^2$ are  well defined on the whole $M$, the formula holds on $M$. From the maximum principle we get that $|H|$ is constant and $|A|^2=m$.
\end{proof}

The case $|A|^2\leq m$ was solved by J.H.~Chen in \cite{C93}. Here we include the proof for two reasons. First, the original one is in Chinese and second, the formalism used by J.H.~Chen was local, while ours is globally invariant. Moreover, the proof we present is slightly shorter.

\begin{theorem}[\cite{C93}]\label{th: jchen1}
Let $\varphi:M^m\to \mathbb{S}^{m+1}$ be a compact hypersurface in $\mathbb{S}^{m+1}$. If $\varphi$ is proper biharmonic and $|A|^2\leq m$, then $\varphi$ is CMC and $|A|^2=m$.
\end{theorem}

\begin{proof}
We may assume that $M$ is orientable, since, otherwise, we consider the double covering $\tilde{M}$ of $M$. This is compact, connected and orientable, and in the given hypotheses $\tilde{\varphi}:\tilde M\to \mathbb{S}^{m+1}$ is proper biharmonic and $|\tilde A|^2\leq m$. Moreover, $\tilde{\varphi}(\tilde{M})=\varphi(M)$.

As $M$ is orientable, we fix a unit global section $\eta\in C(NM)$ and denote $A=A_\eta$ and $f=(\trace A)/m$.
In the following we shall prove that
\begin{eqnarray}\label{eq: fund_ineq_chen1}
&&\frac{1}{2}\Delta \left(|\grad f|^2+\frac{m^2}{8}f^4+f^2\right)+\frac{1}{2}\Div(|A|^2\grad f^2)\leq\nonumber\\
&&\leq\frac{8(m-1)}{m(m+8)}(|A|^2-m) |A|^2 f^2,
\end{eqnarray}
on M, and this will lead to the conclusion.

From \eqref{eq: caract_bih_hipersurf_spheres}(i) one easily gets
\begin{equation}\label{eq: delta_f^2}
\frac{1}{2}\Delta f^2=(m-|A|^2)f^2-|\grad f|^2
\end{equation}
and
\begin{equation}\label{eq: delta_f^4}
\frac{1}{4}\Delta f^4=(m-|A|^2)f^4-3f^2|\grad f|^2.
\end{equation}

From the Weitzenb\"{o}ck formula we have
\begin{equation}\label{eq: Weitz_norm_grad}
\frac{1}{2}\Delta |\grad f|^2=-\langle\trace\nabla^2\grad f,\grad f\rangle-|\nabla\grad f|^2,
\end{equation}
and, since
$$
\trace \nabla^2\grad f=-\grad(\Delta f)+ \ricci(\grad f),
$$
we obtain
\begin{equation}\label{eq: cons_Weitz_norm_grad}
\frac{1}{2}\Delta |\grad f|^2=\langle \grad \Delta f,\grad f\rangle-\ricci(\grad f,\grad f)-|\nabla\grad f|^2.
\end{equation}

Equations \eqref{eq: caract_bih_hipersurf_spheres}(i) and \eqref{eq: delta_f^2} imply
\begin{eqnarray}\label{eq: grad_delta_f}
\langle \grad \Delta f,\grad f\rangle&=&(m-|A|^2)|\grad f|^2-\frac{1}{2}\langle \grad |A|^2, \grad f^2\rangle\nonumber\\
&=&(m-|A|^2)|\grad f|^2-\frac{1}{2}\left(\Div(|A|^2\grad f^2)+|A|^2\Delta f^2\right)\nonumber\\
&=&m|\grad f|^2-\frac{1}{2}\Div(|A|^2\grad f^2)-|A|^2(m-|A|^2)f^2.
\end{eqnarray}

From the Gauss equation of $M$ in $\mathbb{S}^{m+1}$ we obtain
\begin{equation}\label{eq:ricci-minsn}
\ricci(X,Y)=(m-1)\langle X,Y\rangle+\langle A(X),Y\rangle\trace A-\langle A(X), A(Y)\rangle,
\end{equation}
for all $X, Y\in C(TM)$, therefore, by using \eqref{eq: caract_bih_hipersurf_spheres}(ii),
\begin{equation}\label{eq: cons_Gauss}
\ricci(\grad f,\grad f)=\left(m-1-\frac{3m^2}{4}f^2\right)|\grad f|^2.
\end{equation}

Now, by substituting \eqref{eq: grad_delta_f} and \eqref{eq: cons_Gauss} in \eqref{eq: cons_Weitz_norm_grad} and using \eqref{eq: delta_f^2} and \eqref{eq: delta_f^4}, one obtains
\begin{eqnarray*}
\frac{1}{2}\Delta |\grad f|^2&=&\left(1+\frac{3m^2}{4}f^2\right)|\grad f|^2-\frac{1}{2}\Div(|A|^2\grad f^2)\\
&&-|A|^2(m-|A|^2)f^2-|\nabla\grad f|^2\\
&=&-\frac{1}{2}\Delta f^2-\frac{m^2}{16}\Delta f^4-(m-|A|^2)\left(|A|^2-\frac{m^2}{4}f^2-1\right)f^2\\
&&-\frac{1}{2}\Div(|A|^2\grad f^2)-|\nabla\grad f|^2.
\end{eqnarray*}
Hence
\begin{eqnarray}\label{eq: eq_int_1}
&-\frac{1}{2}\Delta \left(|\grad f|^2+\frac{m^2}{8}f^4+f^2\right)-\frac{1}{2}\Div(|A|^2\grad f^2)=\nonumber\\
&=(m-|A|^2)\left(|A|^2-\frac{m^2}{4}f^2-1\right)f^2+|\nabla\grad f|^2.
\end{eqnarray}

We shall now verify that
\begin{equation}\label{eq: fund_ineq1}
(m-|A|^2)\left(|A|^2-\frac{m^2}{4}f^2-1\right)\geq (m-|A|^2)\left(\frac{9}{m+8}|A|^2-1\right),
\end{equation}
at every point of $M$.
Let us now fix a point $p\in M$. We have two cases.\\
{\it Case 1.} If $\grad_p f\neq 0$, then $e_1=({\grad_p f})/{|\grad_p f|}$ is a principal direction for $A$ with principal curvature $\lambda_1=-m f(p)/2$. By considering $e_k\in T_pM$, $k=2,\ldots,m$, such that $\{e_i\}_{i=1}^m$ is an orthonormal basis in $T_pM$ and $A(e_k)=\lambda_k e_k$, we get at $p$
\begin{eqnarray}\label{eq: |A|}
|A|^2&=&\sum_{i=1}^m |A(e_i)|^2=|A(e_1)|^2+\sum_{k=2}^m |A(e_k)|^2=\frac{m^2}{4}f^2+\sum_{k=2}^m \lambda_k^2\nonumber\\
&\geq& \frac{m^2}{4}f^2+\frac{1}{m-1}\left(\sum_{k=2}^m \lambda_k\right)^2=\frac{m^2(m+8)}{4(m-1)}f^2,
\end{eqnarray}
thus inequality \eqref{eq: fund_ineq1} holds at $p$.
\\
{\it Case 2.} If $\grad_p f = 0$, then either there exists an open set $U\subset M$, $p\in U$, such that $\grad f_{/U}=0$, or $p$ is a limit point for the set $V=\{q\in M: \grad_q f\neq 0\}$.\\
In the first situation, we get that $f$ is constant on $U$, and from a unique continuation result for biharmonic maps (see \cite{O03}), this constant must be different from zero. Equation \eqref{eq: caract_bih_hipersurf_spheres}(i) implies $|A|^2=m$ on $U$, and therefore inequality \eqref{eq: fund_ineq1} holds at $p$.\\
In the second situation, by taking into account {\it Case 1} and passing to the limit, we conclude that inequality \eqref{eq: fund_ineq1} holds at $p$.

In order to evaluate the term $|\nabla \grad f|^2$ of equation \eqref{eq: eq_int_1}, let us consider a local orthonormal frame field $\{E_i\}_{i=1}^m$ on $M$. Then, also using \eqref{eq: caract_bih_hipersurf_spheres}(i),
\begin{eqnarray}\label{eq: fund_ineq2}
|\nabla \grad f|^2&=&\sum_{i,j=1}^m\langle \nabla_{E_i}\grad f,E_j\rangle^2\nonumber\geq\sum_{i=1}^m\langle \nabla_{E_i}\grad f,E_i\rangle^2\\
&\geq&\frac{1}{m}\left(\sum_{i=1}^m\langle \nabla_{E_i}\grad f,E_i\rangle\right)^2= \frac{1}{m}(\Delta f)^2\nonumber\\
&=&\frac{1}{m}(m-|A|^2)^2 f^2.
\end{eqnarray}
In fact, \eqref{eq: fund_ineq2} is a global formula.

Now, using \eqref{eq: fund_ineq1} and \eqref{eq: fund_ineq2} in \eqref{eq: eq_int_1}, we obtain \eqref{eq: fund_ineq_chen1}, and by integrating it, since $|A|^2\leq m$, we get
\begin{equation}\label{eq: int3}
(|A|^2-m)|A|^2 f^2=0
\end{equation}
on $M$. Suppose that there exists $p\in M$ such that $|A(p)|^2\neq m$. Then there exists an open set $U\subset M$, $p\in U$, such that $|A|^2_{/U}\neq m$. Equation \eqref{eq: int3} implies that $|A|^2 f^2_{/U}=0$.
Now, if there were a $q\in U$ such that $f(q)\neq 0$, then $A(q)$ would be zero and, therefore, $f(q)=0$.
 Thus $f_{/U}=0$ and, since $M$ is proper biharmonic, this is a contradiction. Thus $|A|^2=m$ on $M$ and $\Delta f=0$, i.e. $f$ is constant and we conclude.
\end{proof}

\begin{remark}
It is worth pointing out that the statement of Theorem~\ref{th: jchen1} is similar in the minimal case: if $\varphi:M^m\to \mathbb{S}^{m+1}$ is a minimal hypersurface with $|A|^2\leq m$, then either $|A|=0$ or $|A|^2=m$ (see \cite{S68}).
By way of contrast, an analog of Proposition~\ref{prop: |B|>m} is not true in the minimal case. In fact, it was proved in \cite{PT83} that if a minimal hypersurface $\varphi:M^3\to \mathbb{S}^{4}$ has $|A|^2>3$, then
$|A|^2\geq 6$.
\end{remark}
Obviously, from Proposition~\ref{prop: |B|>m} and Theorem~\ref{th: jchen1} we get the following result.

\begin{proposition}
Let $\varphi:M^m\to \mathbb{S}^{m+1}$ be a compact hypersurface. If $\varphi$ is proper biharmonic and $|A|^2$ is constant, then $\varphi$ is CMC and $|A|^2=m$.
\end{proposition}

The next result is a direct consequence of Proposition~\ref{prop: |B|>m}.
\begin{proposition}
Let $\varphi:M^m\to \mathbb{S}^{m+1}$ be a compact hypersurface. If $\varphi$ is proper biharmonic and $|H|^2\geq {4(m-1)}/({m(m+8)})$, then $\varphi$ is CMC.
Moreover,
\begin{itemize}
\item[(i)]  if $m\in\{2,3\}$, then $\varphi(M)$ is a small hypersphere $\mathbb{S}^m(1/\sqrt 2)$;
\item[(ii)] if $m=4$, then $\varphi(M)$ is a small hypersphere $\mathbb{S}^4(1/\sqrt 2)$ or a standard product of spheres $\mathbb{S}^3(1/\sqrt 2)\times \mathbb{S}^1(1/\sqrt 2)$.
\end{itemize}
\end{proposition}
\begin{proof}
Taking into account \eqref{eq: |A|}, the hypotheses imply $|A|^2\geq m$.
\end{proof}

For the non-compact case we obtain the following.
\begin{proposition}
Let $\varphi:M^m\to \mathbb{S}^{m+1}$, $m>2$, be a non-compact hypersurface. Assume that $M$ is complete and has non-negative Ricci curvature. If $\varphi$ is proper biharmonic, $|A|^2$ is constant and $|A|^2\geq m$, then $\varphi$ is CMC and $|A|^2=m$. In this case $|H|^2\leq({(m-2)}/{m})^2$.
\end{proposition}
\begin{proof}
We may assume that $M$ is orientable (otherwise, we consider the double covering $\tilde{M}$ of $M$, which is non-compact, connected, complete, orientable, proper biharmonic and with non-negative Ricci curvature; the final result will remain unchanged).  We consider $\eta$ to be a global unit section in the normal bundle $NM$ of $M$ in $\mathbb{S}^{m+1}$.
Then, on $M$, we have
\begin{equation}\label{eq: d1}
\Delta f=(m-|A|^2)f,
\end{equation}
where $f=(\trace A)/m$,
and
\begin{equation}\label{eq: d2}
\frac{1}{2}\Delta f^2=(m-|A|^2)f^2-|\grad f|^2\leq 0.
\end{equation}
On the other hand, as $f^2=|H|^2\leq |A|^2/m$ is bounded, by the Omori-Yau Maximum Principle (see, for example, \cite{Y75}), there exists a sequence of points $\{p_k\}_{k\in \mathbb{N}}\subset M$ such that
$$
\Delta f^2(p_k)>-\frac{1}{k}\qquad\textrm{and}\qquad \lim_{k\to\infty} f^2(p_k)=\sup_M f^2.
$$
It follows that $\displaystyle{\lim_{k\to\infty}}\Delta f^2(p_k)=0$, so $\displaystyle{\lim_{k\to\infty}((m-|A|^2)f^2(p_k))}=0$.

As $\displaystyle{\lim_{k\to\infty} f^2(p_k)=\sup_M f^2>0}$, we get $|A|^2=m$. But from \eqref{eq: d1} follows that $f$ is a harmonic function on $M$. As $f$ is also a bounded function on $M$, by a result of Yau (see \cite{Y75}), we deduce that $f=$ constant.
\end{proof}

\begin{corollary}
Let $\varphi:M^m\to \mathbb{S}^{m+1}$ be a non-compact hypersurface. Assume that $M$ is complete and has non-negative Ricci curvature. If $\varphi$ is proper biharmonic, $|A|^2$ is constant and $|H|^2\geq {4(m-1)}/({m(m+8)})$, then $\varphi$ is CMC and $|A|^2=m$. In this case, $m\geq 4$ and $|H|^2\leq(({m-2})/{m})^2$.
\end{corollary}

\begin{proposition}
Let $\varphi:M^m\to \mathbb{S}^{m+1}$ be a non-compact hypersurface. Assume that $M$ is complete and has non-negative Ricci curvature. If $\varphi$ is proper biharmonic, $|A|^2$ is constant, $|A|^2\leq m$ and $H$ is nowhere zero, then $\varphi$ is CMC and $|A|^2=m$.
\begin{proof}
As $H$ is nowhere zero we consider $\eta=H/|H|$ a global unit section in the normal bundle.  Then, on $M$,
\begin{equation}
\Delta f=(m-|A|^2)f,
\end{equation}
where $f=|H|>0$. As $m-|A|^2\geq 0$ by a classical result (see, for example, \cite[pag.~2]{L06}) we conclude that  $m=|A|^2$ and therefore $f$ is constant.
\end{proof}
\end{proposition}

\subsection{Case 3}
We first present another result of J.H.~Chen in \cite{C93}. In order to do that, we shall need the following lemma.
\begin{lemma}\label{lem: nablaA}
Let $\varphi:M^m\to \mathbb{S}^{m+1}$ be an orientable hypersurface, $\eta$ a unit section in the normal bundle, and put $A_\eta=A$. Then
\begin{itemize}
\item[(i)] $(\nabla A)(\cdot,\cdot)$ is symmetric,
\item[(ii)] $\langle(\nabla A)(\cdot,\cdot),\cdot\rangle$ is totally symmetric,
\item[(iii)] $\trace (\nabla A)(\cdot,\cdot)=m\grad f$.
\end{itemize}
\end{lemma}

\begin{theorem}[\cite{C93}]\label{th: jchen2}
Let $\varphi:M^m\to \mathbb{S}^{m+1}$ be a compact hypersurface. If $\varphi$ is proper biharmonic, $M$ has non-negative sectional curvature and $m\leq 10$, then $\varphi$ is CMC and $\varphi(M)$ is either $\mathbb{S}^{m}(1/\sqrt 2)$, or $\mathbb{S}^{m_1}(1/\sqrt 2)\times \mathbb{S}^{m_2}(1/\sqrt 2)$, $m_1+m_2=m$, $m_1\neq m_2$.
\end{theorem}

\begin{proof}
For the same reasons as in Theorem~\ref{th: jchen1} we include a detailed proof of this result. We can assume that $M$ is orientable (otherwise, as in the proof of Theorem~\ref{th: jchen1}, we work with the oriented double covering of $M$). Fix a unit section $\eta\in C(NM)$ and put $A=A_\eta$ and $f=(\trace A)/m$.

We intend to prove that the following inequality holds on $M$,
\begin{equation}\label{eq: fund_ineq_chen2}
\frac{1}{2}\Delta\left(|A|^2+\frac{m^2}{2}f^2\right)\leq \frac{3m^2(m-10)}{4(m-1)}|\grad f|^2-\dfrac{1}{2}\sum_{i,j=1}^m (\lambda_i-\lambda_j)^2 R_{ijij}.
\end{equation}

From the Weitzenb\"ock formula we have
\begin{equation}\label{eq: Weitz_norm_A}
\frac{1}{2}\Delta |A|^2=\langle \Delta A,A\rangle-|\nabla A|^2.
\end{equation}
Let us first verify that
\begin{eqnarray}\label{eq: DeltaA_aux}
\trace(\nabla^2 A)(X,\cdot,\cdot)=\nabla_X (\trace\nabla A),
\end{eqnarray}
for all $X\in C(TM)$. Fix $p\in M$ and let $\{E_i\}_{i=1}^n$ be a local orthonormal frame field, geodesic at $p$. Then, also using Lemma~\ref{lem: nablaA}(i), we get at $p$,
\begin{eqnarray*}
\trace(\nabla^2 A)(X,\cdot,\cdot)&=&\sum_{i=1}^m (\nabla^2 A)(X,E_i,E_i)=\sum_{i=1}^m (\nabla_X \nabla A)(E_i,E_i)\nonumber\\
&=&\sum_{i=1}^m \{\nabla_X \nabla A(E_i,E_i)-2\nabla A(\nabla_X E_i,E_i)\}=\sum_{i=1}^m \nabla_X \nabla A(E_i,E_i)\nonumber\\
&=&\nabla_X (\trace\nabla A).
\end{eqnarray*}
Using Lemma~\ref{lem: nablaA}, the Ricci commutation formula (see, for example, \cite{B}) and \eqref{eq: DeltaA_aux}, we obtain
\begin{eqnarray}\label{eq: DeltaA}
\Delta A(X)&=&-(\trace\nabla^2 A) (X)=-\trace(\nabla^2 A)(\cdot,\cdot,X)=-\trace(\nabla^2 A)(\cdot,X,\cdot)\nonumber\\
&=&-\trace(\nabla^2 A)(X,\cdot,\cdot)- \trace(RA)(\cdot,X,\cdot)\nonumber\\
&=&-\nabla_X (\trace\nabla A)-\trace (R A)(\cdot,X,\cdot)\nonumber\\
&=& -m\nabla_X \grad f-\trace (R A)(\cdot,X,\cdot),
\end{eqnarray}
where
$$
RA(X,Y,Z)=R(X,Y)A(Z)-A(R(X,Y)Z),\quad \forall\,X,Y,Z\in C(TM).
$$

Also, using \eqref{eq: caract_bih_hipersurf_spheres}(ii) and Lemma~\ref{lem: nablaA}, we obtain
\begin{eqnarray}\label{eq: partial}
\trace\langle A(\nabla_{\cdot}\grad f),\cdot\rangle
&=&\trace \langle \nabla_{\cdot}A(\grad f)-(\nabla A)(\cdot,\grad f),\cdot\rangle\nonumber\\
&=&-\dfrac{m}{4}\trace \langle \nabla_{\cdot} \grad f^2,\cdot\rangle-\langle \trace(\nabla A),\grad f\rangle\nonumber\\
&=&\dfrac{m}{4}\Delta f^2-m|\grad f|^2.
\end{eqnarray}

Using \eqref{eq: DeltaA} and \eqref{eq: partial}, we get
\begin{eqnarray}\label{eq: DeltaA_A}
\langle \Delta A,A\rangle&=&\trace\langle \Delta A(\cdot),A(\cdot)\rangle\nonumber\\
&=&-m\trace\langle \nabla_{\cdot}\grad f,A(\cdot)\rangle+\langle T, A\rangle\nonumber\\
&=&-m\trace\langle A(\nabla_{\cdot}\grad f),\cdot\rangle+\langle T, A\rangle\nonumber\\
&=&m^2|\grad f|^2-\dfrac{m^2}{4}\Delta f^2+\langle T, A\rangle,
\end{eqnarray}
where $T(X)=-\trace (R A)(\cdot,X,\cdot)$, $X\in C(TM)$.

In the following we shall verify that
\begin{equation}\label{eq: estim_nablaA}
|\nabla A|^2\geq\dfrac{m^2(m+26)}{4(m-1)}|\grad f|^2,
\end{equation}
at every point of $M$. Now, let us fix a point $p\in M$.

If $\grad_p f=0$, then \eqref{eq: estim_nablaA} obviously holds at $p$.

If $\grad_p f\neq 0$, then on a neighborhood $U\subset M$ of $p$ we can consider an orthonormal frame field $E_1=({\grad f})/{|\grad f|}$, $E_2$,\ldots, $E_m$, where $E_k(f)=0$, for all $k=2,\ldots, m$.
Using \eqref{eq: caract_bih_hipersurf_spheres}(ii), we obtain on $U$
\begin{eqnarray}\label{eq: NA1}
\langle (\nabla A)(E_1,E_1),E_1\rangle&=&\frac{1}{|\grad f|^3}(\langle \nabla_{\grad f}A(\grad f),\grad f\rangle\nonumber\\&&
-\langle A(\nabla_{\grad f}\grad f),\grad f\rangle)\nonumber\\
&=&-\frac{m}{2}|\grad f|.
\end{eqnarray}
From here, using Lemma~\ref{lem: nablaA}, we also have on $U$
\begin{eqnarray}\label{eq: NA2}
\sum_{k=2}^m\langle (\nabla A)(E_k,E_k),E_1\rangle&=&\sum_{i=1}^m\langle (\nabla A)(E_i,E_i),E_1\rangle-\langle (\nabla A)(E_1,E_1),E_1\rangle\nonumber\\
&=&\langle \trace\nabla A,E_1\rangle+\frac{m}{2}|\grad f|=\frac{3m}{2}|\grad f|.
\end{eqnarray}
Using \eqref{eq: NA1} and \eqref{eq: NA2}, we have on $U$
\begin{eqnarray}
|\nabla A|^2&=&\sum_{i,j=1}^m|(\nabla A)(E_i,E_j)|^2 =\sum_{i,j,h=1}^m\langle(\nabla A)(E_i,E_j),E_h\rangle^2\nonumber\\
&\geq& \langle(\nabla A)(E_1,E_1),E_1\rangle^2+ 3\sum_{k=2}^m\langle(\nabla A)(E_k,E_k),E_1\rangle^2\nonumber\\
&\geq& \langle(\nabla A)(E_1,E_1),E_1\rangle^2+ \frac{3}{m-1}\left(\sum_{k=2}^m\langle(\nabla A)(E_k,E_k),E_1\rangle\right)^2\nonumber\\
&=&\dfrac{m^2(m+26)}{4(m-1)}|\grad f|^2,
\end{eqnarray}
thus \eqref{eq: estim_nablaA} is verified, and \eqref{eq: Weitz_norm_A} implies
\begin{equation}\label{eq: Delta_intermed}
\frac{1}{2}\Delta\left(|A|^2+\frac{m^2}{2} f^2\right)\leq \frac{3m^2(m-10)}{4(m-1)}|\grad f|^2+\langle T,A\rangle.
\end{equation}

Fix $p\in M$ and consider $\{e_i\}_{i=1}^m$ to be an orthonormal basis of $T_pM$, such that $A(e_i)=\lambda_i e_i$. Then, at $p$, we get
\begin{eqnarray*}\label{eq: T_A}
\langle T, A\rangle=-\dfrac{1}{2}\sum_{i,j=1}^m (\lambda_i-\lambda_j)^2 R_{ijij},
\end{eqnarray*}
and then \eqref{eq: Delta_intermed} becomes \eqref{eq: fund_ineq_chen2}.

Now, since $m\leq 10$ and $M$ has non-negative sectional curvature, we obtain
$$
\Delta\left(|A|^2+\frac{m^2}{2}|H|^2\right)\leq 0
$$
on $M$. As $M$ is compact, we have
$$
\Delta\left(|A|^2+\frac{m^2}{2}|H|^2\right)= 0
$$
on $M$, which implies
\begin{equation}\label{eq:lambdarijij}
(\lambda_i-\lambda_j)^2 R_{ijij}=0
\end{equation}
 on $M$. Fix $p\in M$.
From the Gauss equation for $\varphi$, $R_{ijij}=1+\lambda_i\lambda_j$, for all $i\neq j$, and from
\eqref{eq:lambdarijij} we obtain
$$
(\lambda_i-\lambda_j) (1+\lambda_i\lambda_j)=0,\quad i\neq j.
$$
Let us now fix $\lambda_1$. If there exists another principal curvature $\lambda_j\neq \lambda_1$, $j>1$, then from the latter relation we get that $\lambda_1\neq 0$ and  $\lambda_j=-1/\lambda_1$.
Thus $\varphi$ has at most two distinct principal curvatures at $p$. Since $p$ was arbitrarily fixed, we obtain that $\varphi$ has at most two distinct principal curvatures everywhere and we conclude by using Theorem~\ref{th: hypersurf_2curv}.
\end{proof}

\begin{proposition}\label{pro:nonnegricciesistepxpnozero}
Let $\varphi:M^m\to \mathbb{S}^{m+1}$ , $m\geq 3$, be a hypersurface. Assume that $M$ has non-negative sectional curvature and for all $p\in M$ there exists $X_p\in T_pM$, $|X_p|=1$, such that $\ricci(X_p,X_p)=0$. If $\varphi$ is proper biharmonic, then $\varphi(M)$ is an open part of $\mathbb{S}^{m-1}(1/\sqrt 2)\times\mathbb{S}^1(1/\sqrt 2)$.
\end{proposition}

\begin{proof}
Let $p\in M$ be an arbitrarily fixed point, and $\{e_i\}_{i=1}^m$ an orthonormal basis in $T_pM$ such that $A(e_i)=\lambda_i e_i$. For $i\neq j$, using \eqref{eq:ricci-minsn},  we have that $\ricci(e_i,e_j)=0$. Therefore, $\{e_i\}_{i=1}^m$ is also a basis of eigenvectors for the Ricci curvature. Now, if $\ricci(e_i,e_i)>0$ for all $i=1,\ldots m$, then $\ricci(X,X)>0$ for all $X\in T_pM\setminus\{0\}$. Thus there must exist $i_0$ such that $\ricci(e_{i_0},e_{i_0})=0$. Assume that $\ricci(e_{1},e_{1})=0$. From $0=\ricci(e_{1},e_{1})=\sum_{j=2}^m R_{1j1j}=\sum_{j=2}^m K_{1j}$ and since $K_{1j}\geq 0$ for all
$j\geq 2$, we conclude that  $K_{1j}=0$ for all $j\geq 2$, that is  $1+\lambda_1 \lambda _j=0$  for all $j\geq 2$. The latter implies that $\lambda_1\neq 0$ and $\lambda_j=-1/\lambda_1$ for all $j\geq 2$. Thus $M$ has two distinct principal curvatures everywhere, one of them of multiplicity one.
\end{proof}

\begin{remark}
If $\varphi:M^m\to \mathbb{S}^{m+1}$, $m\geq 3$, is  a compact hypersurface, then the conclusion of
Proposition~\ref{pro:nonnegricciesistepxpnozero} holds replacing the hypothesis on the Ricci curvature with the requirement  that the first fundamental group is infinite. In fact, the full classification of compact hypersurfaces
in $\mathbb{S}^{m+1}$ with non-negative sectional curvature and infinite first fundamental group was given in \cite{C03}.
\end{remark}

\section{PMC biharmonic immersions in $\mathbb{S}^n$}

In this section we list some of the most important known results on PMC biharmonic submanifolds in spheres and we prove some new ones. In order to do that we first need the following lemma.

\begin{lemma}\label{lem: AH_B}
Let $\varphi:M^m\to N^n$ be an immersion. Then $|A_H|^2\leq |H|^2 |B|^2$ on $M$. Moreover, $|A_H|^2= |H|^2 |B|^2$ at $p\in M$ if and only if either $H(p)=0$, or the first normal of $\varphi$ at $p$ is spanned by $H(p)$.
\end{lemma}

\begin{proof}
Let $p\in M$. If $|H(p)|=0$, then the conclusion is obvious.
Consider now the case when $|H(p)|\neq0$, let $\eta_p=H(p)/|H(p)|\in N_pM$ and let $\{e_i\}_{i=1}^m$ be a basis in $T_pM$. Then, at $p$,
\begin{eqnarray*}
|A_H|^2&=&\sum_{i,j=1}^m\langle A_H(e_i),e_j\rangle^2=\sum_{i,j=1}^m\langle B(e_i,e_j),H\rangle^2=|H|^2\sum_{i,j=1}^m\langle B(e_i,e_j),\eta_p\rangle^2\\
&\leq& |H|^2 |B|^2.
\end{eqnarray*}
In this case equality holds if and only if $\displaystyle{\sum_{i,j=1}^m\langle B(e_i,e_j),\eta_p\rangle^2=|B|^2},$
i.e.
$$
\langle B(e_i,e_j),\xi_p\rangle =0,\quad \forall\,\xi_p\in N_pM\,\text{ with}\,\, \xi_p\perp H(p).
$$
This is equivalent to the first normal at $p$ being spanned by $H(p)$ and we conclude.
\end{proof}

Using the above lemma we can prove the following lower bound for the norm of the second fundamental form.

\begin{proposition}
Let $\varphi:M^m\to \mathbb{S}^n$ be a PMC proper biharmonic  immersion. Then $m\leq |B|^2$ and equality holds if and only if $\varphi$ induces  a CMC proper biharmonic  immersion of $M$ into a totally geodesic sphere $\mathbb{S}^{m+1}\subset \mathbb{S}^n$.
\end{proposition}
\begin{proof}
By Corollary~\ref{th: caract_bih_pmc} we have $|A_H|^2=m|H|^2$ and, by using Lemma~\ref{lem: AH_B}, we obtain $m\leq|B|^2$.

Since $H$ is parallel and nowhere zero, equality holds if and only if the first normal is spanned by $H$, and we can apply the codimension reduction result of J.~Erbacher (\cite{E71}) to obtain the existence of a totally geodesic sphere $\mathbb{S}^{m+1}\subset \mathbb{S}^n$, such that $\varphi$ is an immersion of $M$ into $\mathbb{S}^{m+1}$. Since $\varphi:M^m\to \mathbb{S}^n$ is PMC proper biharmonic, the restriction $M^m\to \mathbb{S}^{m+1}$ is CMC proper biharmonic.
\end{proof}

\begin{remark}
\begin{itemize}
\item[(i)] Let $\varphi=\imath\circ\phi:M\to \mathbb{S}^n$ be a proper biharmonic  immersion of class {\bf B3}. Then $m\leq|B|^2$ and equality holds if and only if the induced $\phi$ is totally geodesic.

\item[(ii)] Let $\varphi=\imath\circ(\phi_1\times\phi_2): M_1\times M_2\to \mathbb{S}^n$ be a proper biharmonic  immersion of class {\bf B4}. Then $m\leq|B|^2$ and equality holds if and only if both $\phi_1$ and $\phi_2$ are totally geodesic.
\end{itemize}
\end{remark}

The above remark suggests to look for PMC proper biharmonic  immersions with $|H|=1$ and
$|B|^2=m$.

\begin{corollary}
Let $\varphi:M^m\to\mathbb{S}^n$ be a PMC proper biharmonic  immersion. Then $|H|=1$ and $|B|^2=m$ if and only if $\varphi(M)$ is an open part of $\mathbb{S}^{m}(1/\sqrt 2)\subset\mathbb{S}^{m+1}\subset\mathbb{S}^n$.
\end{corollary}

The case when $M$ is a surface is more rigid. Using the classification of PMC surfaces in  $\mathbb{S}^{n}$ given by S.-T.~Yau \cite{Y74}, and \cite[Corollary~5.5]{BMO08}, we obtain the following result.

\begin{theorem}[\cite{BMO08}]\label{th: bih_PMC_surf}
Let $\varphi:M^2\to\mathbb{S}^n$ be a PMC proper biharmonic surface. Then $\varphi$ induces a minimal immersion of $M$ into a small hypersphere $\mathbb{S}^{n-1}(1/\sqrt{2})\subset\mathbb{S}^n$.
\end{theorem}

\begin{remark}
If $n=4$ in Theorem~\ref{th: bih_PMC_surf}, then the same conclusion holds under the weakened assumption that the surface is CMC as it was shown in \cite{BO09}.
\end{remark}
In the higher dimensional case we have the following bounds for the value of the mean curvature of a
PMC proper biharmonic  immersion.

\begin{theorem}[\cite{BO12}]\label{th: pmc1}
Let $\varphi:M^m\to\mathbb{S}^n$ be a PMC proper biharmonic  immersion. Assume that $m>2$ and $|H|\in (0,1)$.
Then $|H|\in (0,({m-2})/{m}]$, and $|H|=({m-2})/{m}$ if and only
if locally $\varphi(M)$ is an open part of a standard product
$$
M_1\times\mathbb{S}^1(1/\sqrt{2})\subset\mathbb{S}^n,
$$
where $M_1$ is a minimal embedded submanifold of $\mathbb{S}^{n-2}(1/\sqrt{2})$. Moreover, if $M$ is
complete, then the above decomposition of $\varphi(M)$ holds globally, where $M_1$ is a complete minimal submanifold of $\mathbb{S}^{n-2}(1/\sqrt{2})$.
\end{theorem}

\begin{remark}
The same result of Theorem~\ref{th: pmc1} was proved, independently, in \cite{WW12}.
\end{remark}
If we assume that $M$ is compact and $|B|$ is bounded we obtain the following theorem.

\begin{theorem}\label{th: pmc-santos}
Let $\varphi:M^m\to\mathbb{S}^{m+d}$ be a compact PMC proper biharmonic  immersion with $m\geq 2$, $d\geq 2$ and
$$
m<|B|^2\leq m \frac{d-1}{2d-3}\left(1+\frac{3d-4}{d-1}|H|^2-\frac{m-2}{\sqrt{m-1}}|H| \sqrt{1-|H|^2}\right).
$$
\begin{itemize}
\item[(i)] If $m=2$, then $|H|=1$, and  either $d=2$, $|B|^2=6$, $\varphi(M^2)=\mathbb{S}^{1}(1/{2})\times\mathbb{S}^{1}(1/{2})\subset\mathbb{S}^{3}(1/\sqrt{2})$ or $d=3$, $|B|^2=14/3$, $\varphi(M^2)$ is the Veronese minimal surface in $\mathbb{S}^{3}(1/\sqrt{2})$.

\item[(ii)] If $m>2$, then $|H|=1$, $d=2$, $|B|^2=3m$ and
$$
\varphi(M^m)=\mathbb{S}^{m_1}\left(\sqrt{{m_1}/{(2m)}}\right)\times \mathbb{S}^{m_2}\left(\sqrt{{m_2}/{(2m)}}\right)\subset \mathbb{S}^{m+1}(1/\sqrt{2}),
$$
where $m_1+m_2=m$, $m_1\geq 1$ and $m_2\geq 1$.
\end{itemize}
\end{theorem}
\begin{proof}
The result follows from the classification of compact PMC immersions with bounded  $|B|^2$ given in
Theorem~1.6 of \cite{S94}.
\end{proof}

\begin{theorem}[\cite{BO12}]\label{th: pmc2}
Let $\varphi:M^m\to\mathbb{S}^n$ be a PMC proper biharmonic  immersion with $\nabla A_H=0$. Assume that $|H|\in (0,({m-2})/{m})$.
Then, $m>4$ and, locally,
$$
\varphi(M)=M^{m_1}_1\times M^{m_2}_2
\subset\mathbb{S}^{n_1}(1/\sqrt{2})\times\mathbb{S}^{n_2}(1/\sqrt{2})\subset\mathbb{S}^n,
$$
where $M_i$ is a minimal embedded submanifold of $\mathbb{S}^{n_i}(1/\sqrt{2})$, $m_i\geq 2$,
$i=1,2$, $m_1+m_2=m$, $m_1\neq m_2$, $n_1+n_2=n-1$. In this case $|H|={|m_1-m_2|}/{m}$.
Moreover, if $M$ is complete, then the above decomposition of $\varphi(M)$ holds globally, where $M_i$ is a complete minimal submanifold of $\mathbb{S}^{n_i}(1/\sqrt{2})$, $i=1,2$.

\end{theorem}

\begin{corollary}
Let $\varphi:M^m\to\mathbb{S}^n$, $m\in\{3,4\}$, be a PMC proper biharmonic  immersion with $\nabla A_H=0$. Then $|H|\in \{({m-2})/{m},1\}$. Moreover, if $|H|=({m-2})/{m}$, then locally
$\varphi(M)$ is an open part of a standard product
$$
M_1\times\mathbb{S}^1(1/\sqrt{2})\subset\mathbb{S}^n,
$$
where $M_1$ is a minimal embedded submanifold of $\mathbb{S}^{n-2}(1/\sqrt{2})$,
and if $|H|=1$, then $\varphi$ induces a minimal immersion of $M$ into $\mathbb{S}^{n-1}(1/\sqrt 2)$.
\end{corollary}

We should note that there exist examples of proper biharmonic submanifolds of $\mathbb{S}^{5}$ and
$\mathbb{S}^{7}$ which are not PMC but with $\nabla A_H=0$ (see \cite{S05} and \cite{FO12}).

\section{Parallel biharmonic  immersions in $\mathbb{S}^n$}

An immersed submanifold is said to be {\it parallel} if
its second fundamental form $B$ is parallel, that is $\nabla^\perp B=0$.

In the following we give the classification for proper biharmonic parallel immersed surfaces in $\mathbb{S}^n$.
\begin{theorem}\label{teo:parallel-surfaces}
Let $\varphi:M^2\to \mathbb{S}^n$ be a parallel surface in $\mathbb{S}^n$. If $\varphi$ is proper biharmonic, then the codimension can be reduced to $3$ and $\varphi(M)$ is an open part of either
\begin{itemize}
\item[{\rm(i)}] a totally umbilical sphere $\mathbb{S}^2(1/\sqrt2)$ lying in
a totally geodesic $\mathbb{S}^3\subset \mathbb{S}^5$,
or
\item[{\rm(ii)}] the minimal flat torus $\mathbb{S}^1(1/2)\times \mathbb{S}^1(1/2)\subset
\mathbb{S}^3(1/\sqrt2)$; $\varphi(M)$ lies in a totally geodesic $\mathbb{S}^4\subset \mathbb{S}^5$,
or
\item[{\rm(iii)}] the minimal Veronese surface in $\mathbb{S}^4(1/\sqrt2)\subset \mathbb{S}^5$.
\end{itemize}
\end{theorem}
\begin{proof}
The proof relies on the fact that parallel submanifolds in $\mathbb{S}^n$ are classified in the following three categories (see, for example, \cite{C10}):
\begin{itemize}
\item[(a)] a totally umbilical sphere $\mathbb{S}^2(r)$  lying in a totally geodesic $\mathbb{S}^3\subset\mathbb{S}^n$;

\item[(b)] a flat torus lying in a totally geodesic $\mathbb{S}^4\subset\mathbb{S}^n$
defined by
$$
(0,\ldots,0,a\cos u,a\sin u,b\cos v,b\sin v,\sqrt{1-a^2 -b^2}),\quad a^2 +b^2 \leq1;
$$
\item[(c)] a surface of positive constant curvature lying in a totally geodesic $\mathbb{S}^5\subset\mathbb{S}^n$ defined by
$$
r\left(0,\ldots,0,\frac{v w}{\sqrt{3}},\frac{u w}{\sqrt{3}},\frac{u v}{\sqrt{3}},\frac{u^2-v^2}{2\sqrt{3}},
\frac{u^2+v^2-2w^2}{6},\frac{\sqrt{1-r^2}}{r}\right),
$$
with $u^2+v^2+w^2=3$ and $0<r\leq 1$.
\end{itemize}
In case (a) the biharmonicity implies directly (i). Requiring the immersion in (b) to be biharmonic and using
\cite[Corollary~5.5]{BMO08} we get that $\sqrt{a^2 +b^2}=1/2$ and then (ii) follows. The immersion in (c) induces a minimal immersion of the surface in the hypersphere $\mathbb{S}^4(r)\subset\mathbb{S}^5$. Then, applying \cite[Theorem~3.5]{CMO01}, the immersion in (c) reduces to that in (iii).
\end{proof}

In all  three  cases of Theorem~\ref{teo:parallel-surfaces}, $\varphi$ is of type {\bf B3} and thus its mean curvature is $1$. In the  higher dimensional case we know, from Theorem~\ref{teo:h=cst-b3}, that if $|H|=1$,
then $\varphi$ is of type {\bf B3}. Moreover, if we assume that $\varphi$ is also parallel, then the induced minimal immersion in $\mathbb{S}^{n-1}(1/\sqrt{2})$ is parallel as well.

If $\nabla^{\perp}B=0$, then $\nabla^{\perp}H=0$ and $\nabla A_H=0$. Therefore Theorem~\ref{th: pmc1} and Theorem~\ref{th: pmc2} hold also for parallel proper biharmonic  immersions in $\mathbb{S}^n$. From this and Theorem~\ref{teo:parallel-surfaces}, in order to classify all parallel proper biharmonic  immersions in $\mathbb{S}^{n}$, we are left with the case when $m>2$ and $|H|\in (0,1)$.

\begin{theorem}
Let $\varphi:M^m\to\mathbb{S}^n$ be a parallel proper biharmonic  immersion. Assume that $m>2$ and $|H|\in (0,1)$. Then $|H|\in (0,({m-2})/{m}]$. Moreover:
\begin{itemize}
\item[(i)] $|H|=({m-2})/{m}$ if and only
if locally $\varphi(M)$ is an open part of a standard product
$$
M_1\times\mathbb{S}^1(1/\sqrt{2})\subset\mathbb{S}^n,
$$
where $M_1$ is a parallel minimal embedded submanifold of $\mathbb{S}^{n-2}(1/\sqrt{2})$;

\item[(ii)] $|H|\in (0,({m-2})/{m})$ if and only if $m>4$ and, locally,
$$
\varphi(M)=M^{m_1}_1\times M^{m_2}_2
\subset\mathbb{S}^{n_1}(1/\sqrt{2})\times\mathbb{S}^{n_2}(1/\sqrt{2})\subset\mathbb{S}^n,
$$
where $M_i$ is a parallel minimal embedded submanifold of $\mathbb{S}^{n_i}(1/\sqrt{2})$, $m_i\geq 2$,
$i=1,2$, $m_1+m_2=m$, $m_1\neq m_2$, $n_1+n_2=n-1$.
\end{itemize}
\end{theorem}
\begin{proof} We only have to prove that $M_i$ is a parallel minimal submanifold of $\mathbb{S}^{n_i}(1/\sqrt{2})$, $m_i\geq 2$. For this, denote by $B^i$ the second fundamental form of $M_i$ in $\mathbb{S}^{n_i}(1/\sqrt{2})$, $i=1,2$. If $B$ denotes the second fundamental form of $M_1\times M_2$ in $\mathbb{S}^n$, it is easy to verify, using the expression of the second fundamental form of  $\mathbb{S}^{n_1}(1/\sqrt{2})\times\mathbb{S}^{n_2}(1/\sqrt{2})$  in $\mathbb{S}^n$, that
$$
(\nabla^\perp_{(X_1,X_2)} B)((Y_1,Y_2),(Z_1,Z_2))=((\nabla^\perp_{X_1} B^1)(Y_1,Z_1),(\nabla^\perp_{X_2} B^2)(Y_2,Z_2)),
$$
for all $X_1,Y_1,Z_1\in C(TM_1)$, $X_2 ,Y_2, Z_2\in C(TM_2)$. Consequently, $M_1\times M_2$ is parallel in $\mathbb{S}^n$ if and only if $M_i$ is parallel in $\mathbb{S}^{n_i}(1/\sqrt{2})$, $i=1,2$.
\end{proof}

\section{Open problems}
We list some open problems and conjectures that seem to be natural.

\begin{conjecture}\label{conj: 1}
The only proper biharmonic hypersurfaces in $\mathbb{S}^{m+1}$ are the open parts of hyperspheres $\mathbb{S}^m(1/\sqrt2)$ or of the standard products of spheres $\mathbb{S}^{m_1}(1/\sqrt2)\times\mathbb{S}^{m_2}(1/\sqrt2)$,
$m_1+m_2=m$, $m_1\neq m_2$.
\end{conjecture}

Taking into account the results presented in this paper, we have a series of statements equivalent to Conjecture~\ref{conj: 1}:
\begin{itemize}
\item[1.] A proper biharmonic hypersurface in $\mathbb{S}^{m+1}$ has at most two principal curvatures everywhere.
\item[2.] A proper biharmonic hypersurface in $\mathbb{S}^{m+1}$ is parallel.
\item[3.] A proper biharmonic hypersurface in $\mathbb{S}^{m+1}$ is CMC and has non-negative sectional curvature.
\item[4.] A proper biharmonic hypersurface in $\mathbb{S}^{m+1}$ is isoparametric.
\end{itemize}

One can also state the following intermediate conjecture.

\begin{conjecture}\label{conj: 2}
The proper biharmonic hypersurfaces in $\mathbb{S}^{m+1}$ are  CMC.
\end{conjecture}

Related to PMC immersions and, in particular, to Theorem~\ref{th: pmc2}, we propose the following problem.

\begin{problem}
Find a PMC proper biharmonic  immersion $\varphi:M^m\to\mathbb{S}^{n}$ such that
$A_{H}$ is not parallel.
\end{problem}

\end{document}